\documentclass{amsart}
\usepackage{amsmath}
\usepackage{amssymb}
\usepackage{graphicx}

\newtheorem{theorem}{Theorem}[section]

\newtheorem{lemma}[theorem]{Lemma}

\newtheorem{observation}[theorem]{Observation}

\newtheorem{corollary}[theorem]{Corollary}

\newtheorem{proposition}[theorem]{Proposition}

\newtheorem{assertion}[theorem]{Assertion}
\newtheorem{conjecture}[theorem]{Conjecture}
\newtheorem{problem}[theorem]{Problem}

\theoremstyle{definition}  

\newtheorem{definition}[theorem]{Definition}

\newtheorem{example}[theorem]{Example}

\newcommand{\ce}{\mathcal E}

\newcommand{\hetz}{{^\curvearrowright}}

\newcommand{\Cr}[1]{Conjecture~\ref{#1}}
\newcommand{\Lr}[1]{Lemma~\ref{#1}}
\newcommand{\Tr}[1]{Theorem~\ref{#1}}
\newcommand{\Sr}[1]{Section~\ref{#1}}

\newcommand{\N}{\mathbb N}
\newcommand{\R}{\mathbb R}

\newcommand{\Debug}{1}

\ifnum \Debug = 1 
\else  \fi

\newcommand{\showFig}[4]{
   \begin{figure}[htbp]
   \centering
   \noindent
   \includegraphics[width=#2 \linewidth]{#1}
   \caption{\small #4}
   \label{#3}
   \end{figure}
}

\newcommand{\wcr}{finite-cut-respecting}
\newcommand{\scr}{cut-respecting}
\newcommand{\siw}{\ensuremath{\sigma_w}}
\newcommand{\sis}{\ensuremath{\sigma}}
\newcommand{\tam}{\ensuremath{\tau_m}}
\newcommand{\taw}{\ensuremath{\tau_w}}
\newcommand{\tas}{\ensuremath{\tau_s}}
\newcommand{\comment}[1]{}

\begin{document}




\title{The Max-Flow Min-Cut Theorem for Countable  Networks}


\author{Ron Aharoni}
\address{Department of Mathematics\\Technion, Haifa\\ Israel 32000}
\email[Ron Aharoni]{ra@tx.technion.ac.il}
\thanks{The research of the first author was
supported by grants from the Israel Science Foundation, BSF, the M. \&
M.L Bank Mathematics Research Fund and the fund for the promotion
of research at the Technion and a Seniel Ostrow Research Fund}

\author{Eli Berger}
\address{Department of Mathematics\\ Faculty of Science and Science Education\\ Haifa University\\ Israel 32000} \email[Eli Berger]{berger@cri.haifa.ac.il}
\thanks{The research of the second author was
supported by a BSF grant}

\author{Agelos Georgakopoulos}
\address{Universit\"at Hamburg}
\email[Agelos Georgakopoulos]{georgakopoulos@math.uni-hamburg}
\thanks{The research of the first, third and fifth authors was supported by a GIF grant.}

\author{Amitai Perlstein}\address{Department of Mathematics\\Technion, Haifa\\ Israel 32000}
\email[Ron Aharoni]{perlstein@tx.technion.ac.il}

\author{Philipp Spr\"ussel}
\address{Universit\"at Hamburg}
\email[Philipp Spr\"ussel]{spruessel@math.uni-hamburg}




\begin{abstract}
We prove a strong version of  the Max-Flow Min-Cut theorem
for countable networks,
 namely that in every such network there exist a flow and a cut that are ``orthogonal" to each other,
 in the sense that the flow saturates the cut and is zero on the reverse cut.
  If the network does not contain infinite trails then
  this flow can be chosen to be  {\it mundane}, i.e.\ to be a sum of flows along finite paths.
  We show that in the presence of infinite trails there may be no orthogonal pair of a cut and a mundane flow.
   We finally show that for locally finite networks there is an orthogonal pair of a cut and a flow that satisfies Kirchhoff's first law also for ends.
\end{abstract}
\maketitle

\pagestyle{myheadings}
\markboth{\footnotesize\sc R.~Aharoni, E.~Berger, A.~Georgakopoulos, A.~Perlstein,
and P.~Spr\"ussel}{\footnotesize\sc The Max-Flow Min-Cut Theorem for Countable Networks}

\section{Introduction}

Recently, the first two authors of this paper  proved the
following generalisation of Menger's theorem to the infinite case
\cite{aharoniberger}:

\begin{theorem} \label{inf-menger}
Given a possibly infinite digraph and two vertex sets $A$ and $B$
in it, there exists a set $P$ of vertex-disjoint $A$-$B$ paths and
an $A$-$B$-separating set of vertices $S$, such that $S$ consists
of a choice of precisely one vertex from each path in $P$.
\end{theorem}

In the finite case, the closely related edge version of Menger's
theorem can be viewed as the integral version of the Max-Flow
Min-Cut (MFMC) theorem. In fact, the MFMC theorem can easily be
reduced to Menger's theorem, while the standard proofs of the
MFMC theorem yield also its integral version, namely the edge version of Menger's
theorem.

Thus it is natural to ask also for a generalisation of the MFMC
theorem to the infinite case. \Tr{inf-menger}, which was
originally conjectured by Erd\H{o}s, suggests a possible
generalisation. In the language of Linear Programming, the
infinite version of Menger's theorem is formulated in terms of the
complementary slackness conditions, rather than  equality of the
values of dual programs. In the case of the MFMC theorem this leads
 to the conjecture that in any network there exists an {\it orthogonal pair} of a flow and a cut, i.e.\ a flow and a cut related to each other by the complementary slackness conditions. These are tantamount to the
  demands that every edge of the cut is saturated by the flow, and
on each edge of the reverse cut the value of the flow is zero (see \Sr{defs} for precise definitions). The conjecture is thus:

\begin{conjecture}\label{mfmcCheat}
In any (possibly infinite) network  there exists an orthogonal
pair of a flow and a cut.
\end{conjecture}

In this paper we prove this conjecture for countable networks (\Sr{sec:main}).
An important observation is that Conjecture \ref{mfmcCheat}, even in its integral version, does not generalise \Tr{inf-menger}, since the flow may contain infinite paths. This naturally raises the question whether \Cr{mfmcCheat} is true for flows that do not allow any flow to escape to infinity. We call such flows {\it mundane} (see \Sr{sec:mundane} for the precise definition).
As we shall see, this stronger version of Conjecture \ref{mfmcCheat} is false, since even in locally finite networks mundane flows do not necessarily attain the supremum of their values. We thus seek to relax the constraint that no flow escapes to infinity, and are led to two new types of flows: {\it \wcr} flows, which are allowed to send flow to infinity but any amount flowing into an end of the graph must flow out of the same end, and {\it \scr} flows, which are \wcr\ flows with the additional constraint that if a flow circumvents some cut $F$ by flowing through an end, then this circumvention does not exceed the amount that could in principle flow through $F$ (see \Sr{sec:ends} for precise definitions).

In the case of locally finite networks we show, for each of these types of flows, that the infimum of the capacities of all relevant cuts equals the supremum of the values of the corresponding flows (Sections~\ref{sec:mundane} and~\ref{sec:ends}). We then prove that in any  locally finite network there is an orthogonal pair of a \scr\ flow and a cut of minimum capacity (\Sr{sec:ends}).

\section{Definitions and Notation}
\label{defs}

We shall mostly follow the terminology of \cite{diestelBook05}.
Deviations will be explicitly indicated. By $\R_+$ we mean the set
of non-negative real numbers. Referring to a ``function" we shall
mean, unless otherwise stated, that its range is the set of reals.
For a function $f$ and a subset $A$ of its domain we shall write
$f[A]$ for $\sum_{a \in A}f(a)$ (which might be $\infty$).

The characteristic function of a subset $T$ of a set $S$ is
denoted by $\chi_S(T)$, or simply $\chi(T)$ if the identity of $S$ is clear from the context.

For a directed edge $e=(u,v)$ we shall write $u=init(e)$, $v=ter(e)$.
For a vertex $v$ in a digraph we denote by $OUT(v)$ the set of
edges $e$ with $init(e)=v$, and by $IN(v)$ the set of edges $e$ with $ter(e)=v$.


\begin{definition}\label{network}
A {\it network} $\Delta$ is a quadruple $(D,c,s,t)$, where
 $D=(V,E)$ is a digraph with no loops, $c$ is a
function  (called \emph{capacity}) from $E$ to $\R_+$, and $s,t$  are vertices of
$D$, called \emph{source} and \emph{sink} respectively. We shall assume that
$IN(s)=OUT(t)=\emptyset$.
\end{definition}

Throughout this section we shall consider a fixed network $\Delta=(D,c,s,t)$.

For a function $g$ on $E$ and an edge $(u,v)\in E$  we abbreviate $g((u,v))$ to $g(u,v)$. 
For a vertex $v\in V$ we write $d^-_g(v)=g[IN(v)]$,~
$d^+_g(v)= g[OUT(v)]$, and $d_g(v)= d^+_g(v)-d^-_g(v)$. Here
we adopt the convention $\infty -\infty = 0$. Given a
function $f$ on the edge set of an {undirected} graph, the
{\it degree} $d_f(v)$ of a vertex $v$ is the sum of $f(e)$ over
all edges $e$ incident with $v$.

\begin{definition}\label{sinkandkir}
Given a function $f$ on $E$, the set of vertices $x\in V$ for
which $d_f^+(x)= 0$  is denoted by $SINK(f)$.  The set of vertices
$x$ for which $d_f(x)=0$ (and thus $d_f^+(x)= d_f^-(x)$)  is
denoted by $KIR(f)$ ($KIR$ standing for ``Kirchhoff").
\end{definition}

\begin{definition}\label{flow}
A function $f:E \to \R_+$ is called a {\it flow} if:
\begin{itemize}\item (Capacity constraint:) $f(e)\leq c(e)$ for every $e\in E$.
\item (Flow conservation:) $V
\setminus \{s,t\}\subseteq KIR(f)$.
\end{itemize}
\end{definition}
The support of a non-negative function $f$ on $E$, namely the set
of edges $e$ for which $f(e)>0$,  is denoted by $supp(f)$.
The {\it value} $|f|$ of a flow $f$ is defined by $|f|:=d^+_f(s)$.
Note that in infinite networks this is not necessarily equal to
$d_f^-(t)$. If $\vec{C}$ is a directed cycle in $D$ then we say
that $\vec{C}$ is a {\it cycle of $f$} if $f(e)>0$  for every edge
$e \in E(\vec{C})$.

A {\it cut} is a set of edges of the form $E(S,
V\setminus S)$ for some $S\subseteq V$,
where $E(X,Y)$ is the set of edges directed from $X$ to
$Y$. An $s$--$t$~cut is a cut $E(S,V \setminus S)$ such that $s \in S$ and $t \not\in S$.
A flow $f$ is said to {\it saturate} an edge $e$ if $f(e)=c(e)$.
It is said to  saturate a set $F$ of edges if it saturates all
edges in $F$.
A flow $f$ and an $s$--$t$~cut $E(S,V \setminus S)$ are {\it
orthogonal} to each other if $f$ saturates $E(S,V \setminus S)$ and is zero on
every edge in $E(V \setminus S, S)$.

A $1$-way infinite path in an undirected graph $G$ 
is called a {\it ray}. Two rays $R,L$ in $G$ are {\it equivalent} 
if no finite set of vertices separates them. The corresponding equivalence classes of rays are the {\it ends} of $G$.



\section{A vertex version}
\label{sec:vx}

As already mentioned, in the finite case the edge version of
Menger's theorem is just the integral case of the MFMC theorem,
namely the case  in which the capacity function is identically
$1$ and the desired flow only takes the values $0$ and $1$. The vertex
version and the edge version of Menger's theorem are easily
derivable from each other. To get the vertex version from the edge
version, one splits each vertex into a ``receiving" copy and an
``emitting" copy, connected by an edge. We do not elaborate more on this 
transformation since it is not needed for our results. 

The other direction of the equivalence, namely the derivation of the edge
version from the vertex version is done by a transformation that will be 
described here in more details and in a more general context, allowing it to
be used also in the non-integral case, yielding an equivalent
version of \Cr{mfmcCheat}. To state it we need the
following definitions:

\begin{definition}
A {\it weighted web} $\Gamma$ is a quadruple $(D,A,B,w)$, where
$D$ is a digraph, $A,B \subseteq V(D)$ and $w$ is a
function from $V(D)$ to $\R_+$. Let $V(\Gamma)=V(D)$ and $E(\Gamma)=E(D)$.
\end{definition}

Let $\Gamma=(D,A,B,w)$ be a weighted web fixed throughout this section.

\begin{definition}
A {\it current} $f$ in $\Gamma$ is a function
 from $E(D)$ to $\R_+$ such that
\begin{enumerate}
\item
  $d_f^+(x) \le w(x)$ and $d_f^-(x) \le w(x)$ for every vertex
$x \in V(D)$;

\item
  $d_f^+(x) \le  d_f^-(x)$ for every vertex $x \in V(D)
\setminus
 A$; and

 \item
 $d_f^-(a)=0$ for every $a \in A$ and $d_f^+(b) =0$ for every
 $b \in B$.
\end{enumerate}
  A vertex $x$ is said to be {\it saturated} by $f$ if $x \in A$
  or
$d_f^-(x)=w(x)$. The set of  vertices that are saturated by $f$ is
denoted by $SAT(f)$.  The set $SAT(f) \cap SINK(f)$ is denoted by
$TER(f)$ (standing for ``terminal points"; recall that $SINK(f)$
is the set of vertices $x$ for which $d_f^+(x)= 0$).
\end{definition}

\begin{definition}\label{webflow}
A current $f$ satisfying  $KIR(f)  \supseteq V(D) \setminus (A
\cup B)$ is called a  ${\it web-flow}$.
\end{definition}

A set $S$ of vertices in $\Gamma$ is said to be $A$-$B$-{\it
separating} (or simply {\em separating}) if every path
from $A$ to $B$ meets $S$. Given a (not necessarily separating)
subset $S$ of $V(D)$, a vertex $x \in S$ is said to be {\em
essential (for separation)} in $S$ if it is not separated from $B$
by $S \setminus \{x\}$. The set of essential elements of $S$ is
denoted by $\ce(S)$. If $S=\ce(S)$ then we say that $S$ is {\em
essential}. It is easy to show:

\begin{lemma}[\cite{aharoniberger}]
  \label{essentialpoints}
  If $S$ is separating, then so is $\ce(S)$.
\end{lemma}

For a set $S$ of vertices in $\Gamma$ we denote by
$RF(S)=RF_\Gamma(S)$ the set of vertices $v$ separated by $S$ from
$B$, namely such that every path from $v$ to $B$ meets $S$. (The
letters ``$RF$'' stand for ``roofed'', a term originating in the
way the authors draw their weighted webs, with the ``$A$" side at
the bottom, and the ``$B$" side on top.) In particular, $S
\subseteq RF(S)$. We define $RF^\circ(S):=RF(S)\setminus \ce(S)$.

Given a current $f$, we write $RF(f)=RF(TER(f))$ and
$RF^\circ(f)=RF^\circ(TER(f))$.

\begin{definition}
Let $f$ be a web flow and let $S$ be a separating set. We say that
$S$ is {\em orthogonal} to $f$ if $S \subseteq SAT(f)$ and $f(u,v)
= 0$ for every pair of vertices $u,v$ with $v \in RF^\circ(S)$ and
$u \in V \setminus RF^\circ(S)$.
\end{definition}

An equivalent conjecture to \Cr{mfmcCheat} is:

\begin{conjecture}\label{vertexmfmc}
In every weighted web there exists a web-flow $f$ and an
$A$-$B$ separating set orthogonal to $f$.
\end{conjecture}

The transformation used
to deduce \Cr{mfmcCheat} from \Cr{vertexmfmc} is the following.
Let $\Delta = (D,c,s,t)$ 
be a network, with notation as in
Definition~\ref{network}. Let $\Gamma=(D',A,B,w)$ be the web
defined by $V(\Gamma)=E(\Delta)$, ~$E(\Gamma) = \{((x,y),(y,z))
\mid (x,y),(y,z) \in E(\Delta)\}$, $A=E(s,V(\Delta)\setminus
\{s\})$, $B=E(V(\Delta)\setminus \{t\},t)$, and $w(e)=c(e)$ for
every $e \in V(\Gamma)=E(\Delta)$.

Clearly, every essential $A$-$B$-separating set of vertices in
$\Gamma$ is also an $s$-$t$~cut in $\Delta$. If $f$ is a
web-flow in $\Gamma$, we can define a flow $g$ in $\Delta$ as
$g(e) = max(d^+_f(e), d^-_f(e))$ (recall, however, that $d^+_f(e)
= d^-_f(e)$ if $e \not\in A\cup B)$. It is straightforward to
check that $g$ is
indeed a flow. 
Moreover, if $f$ is
orthogonal to some $A$-$B$-separating set $S$ of vertices then
it is also orthogonal to $\ce(S)$, and $g$ is orthogonal to the
corresponding cut.

In the following sections we will prove \Cr{vertexmfmc}, and thus \Cr{mfmcCheat}, for the countable case (see \Tr{main}).

\section{Linkability in weighted webs, waves, and an equivalent
conjecture}

In this section we develop some tools that we will use for the proof of \Cr{vertexmfmc} for countable weighted webs. These are generalisations of fundamental notions in the proof of \Tr{inf-menger} in \cite{aharoniberger}, and they could turn out useful in proving the general case of \Cr{vertexmfmc}.

Let $\Gamma=(D,A,B,w)$ be a weighted web fixed throughout this section.

\begin{definition}\label{linkability}
A web-flow $f$ in $\Gamma$ is called a {\it linkage}
if $d^+_f(a)=w(a)$ for every $a \in A$. If a weighted web contains
a linkage it is called {\it linkable}.
\end{definition}




\begin{definition}\label{wave}
A current $f$ in $\Gamma$ is called a {\it wave} if $TER(f)$ is
$A$-$B$-separating and $d_f^+(x)=0$ for all $x\notin RF(f)$.
\end{definition}

If $f,g$ are waves, we write $f\leq g$ if $f(e)\leq g(e)$ for every edge $e$.

\begin{lemma}\label{supwave}
Let $I$ be a totally ordered set, and let $(f_i \mid i \in I)$
 be waves such that
 $f_i \le f_j$ whenever $i \le j$. Then $f=\sup(f_i \mid i \in I)$ is a wave.
\end{lemma}

\begin{proof} Let $P$ be an $A$-$B$ path. Clearly, for $j
\ge i$ we have $SINK(f_i) \supseteq SINK(f_j)$ and $SAT(f_i)
\subseteq SAT(f_j)$. Since $P$ is finite, this means that there
exists an $i$ such that for every $j \ge i$ we have $SINK(f_i)\cap
V(P) = SINK(f_j) \cap V(P)$ and $SAT(f_i)\cap V(P) = SAT(f_j)\cap
V(P)$. Then, $SINK(f)\cap V(P)= SINK(f_i) \cap V(P)$ and
$SAT(f)\cap V(P) \supseteq SAT(f_i)\cap V(P)$. Hence, $TER(f)\cap
V(P)\supseteq TER(f_i)\cap V(P)$, and since $f_i$ is a wave, this
implies that $TER(f)\cap V(P) \neq \emptyset$. This proves that
$f$ is a wave.
\end{proof}


By Zorn's lemma this implies:

\begin{corollary}\label{maxwave}
In  every weighted web there exists a ($\le$)-maximal wave.
\end{corollary}




\begin{definition}
A wave $f$ is called a {\it hindrance} if there exists a vertex
 $a \in A \setminus \ce(TER(f))$ such that $d^+_f(a) <
w(a)$. If $0< \varepsilon < w(a)-d^+_f(a)$ then $f$ is said to be
a $(>\varepsilon)$-hindrance. A weighted web is called {\it
hindered} (respectively $(>\varepsilon)${\it -hindered}) if it contains a hindrance (respectively a $(>\varepsilon)$-hindrance). A weighted web is called
{\it loose} if it contains no non-zero wave and the zero wave is
not a hindrance.
\end{definition}

The following is an easy consequence of the definitions.

\begin{observation}\label{deletingaisgoodforyou}
Let $\Gamma = (D,A,B,w)$ and $\Gamma' = (D,A,B,w')$ be weighted webs such
that  $w'(v) = w(v)$ for all $v \in V \setminus A$ and
 $w'(a) \leq w(a)$ for all $a \in A$.
Then every wave in $\Gamma'$ is a wave in $\Gamma$.
Thus if $\Gamma$ is loose
then so is $\Gamma'$.
\end{observation}

\begin{definition} Let $f$ be a wave. A wave $g$ is called a
{\it trimming} of $f$ if
\begin{enumerate}
\item
$g \le f$
\item
 $RF^\circ(f) \subseteq KIR(g)\cup A$ and:
 \item $TER(g)\setminus A=\ce(TER(f))\setminus A$.
 \end{enumerate}
A wave is called {\em trimmed} if it is a trimming of itself.
 \end{definition}

\begin{lemma}\label{trimming}
Every wave has a trimming. 
\end{lemma}

\begin{proof}

Let $f$ be a wave that is not trimmed, let $x \in RF^\circ(f)
\setminus (KIR(f)\cup A)$
and let $f_1$ be the wave obtained from $f$ by decreasing the values
on $IN(x)$ so that $d_f^-(x)=d_f^+(x)$. One can easily see that
$\ce(TER(f_1)) = \ce(TER(f))$, which means that $f_1$ is indeed a
wave. If $f_1$ is trimmed, we are done. If not, we can find in a
similar way a wave $f_2 \leq f_1$ with $\ce(TER(f_2)) =
\ce(TER(f))$. We can continue this way. 
Note that the sequence of waves obtained this way 
is $\leq$-decreasing, and therefore one can take limits of it. 
So, for example, $f_\omega = \lim_{i < \omega} f_i$ and one can 
check that $f_\omega$ is a wave with $\ce(TER(f_\omega)) = \ce(TER(f))$.
Continuing this process, if necessary,
transfinitely,  we obtain a trimmed wave $f_\alpha$, which is then a
trimming of $f$.
\end{proof}

\begin{definition}\label{quotient}
If $\Gamma$ is a weighted web and $f$ is a wave in $\Gamma$, we
write $\Gamma/f$ for the web $\Xi$ defined by
$A_\Xi=\ce(TER(f))$, $B_\Xi=B$,
$V_\Xi =  V  \setminus RF^\circ(f)$,
$D_\Xi = D[V_\Xi]$ (the subgraph of $D$ induced on $V_\Xi$) and
$w_\Xi=w \upharpoonright V_\Xi$.
\end{definition}

Waves can be combined, as follows:

\begin{definition}
Let $f$ be a wave and $g$ be a current. We denote by $f \hetz g$
the function \nolinebreak {$f+(g \upharpoonright E(\Gamma/f))$.}
\end{definition}

It is easy to check that $f \hetz g$ is a current. In fact, if
$g$ is a wave, then $g\upharpoonright (\Gamma/f)$ is a wave in
$\Gamma/f$, and thus $f \hetz g$ is a wave, which follows from:

\begin{lemma}\label{hetziswave}
If $g\upharpoonright (\Gamma/f)$ is a wave, then $TER(f \hetz g)
\supseteq \ce(TER(f) \cup TER(g))$.
\end{lemma}

\begin{proof}
Let $x \in \ce(TER(f) \cup TER(g))$. We wish to show that $x \in
TER(f \hetz g)$.

We first note that $x$ cannot lie in $RF^\circ(f)$ or
$RF^\circ(g)$, so $x \in SINK(f) \cap SINK(g)$, and hence $x \in
SINK(f \hetz g)$. It remains to show that $x \in SAT(f \hetz g)$.
Since $f \hetz g \geq f$ we have $SAT(f \hetz g) \supseteq
SAT(f)$. Hence we are done in the case $x \in TER(f)$ and we may
assume $x \in TER(g) \setminus TER(f)$.

Since $x \notin TER(f)$ and $x \notin RF^\circ(f)$, we have $x
\notin RF(f)$ and thus, for an edge $e$ entering $x$ we have $e
\in E(\Gamma/f)$ and thus $(f \hetz g)(e) = g(e)$. Since $x \in
SAT(g)$ this yields $x \in SAT(f \hetz g)$, completing the proof.
\end{proof}

One special case that will be of interest is that of bipartite webs.

\begin{definition}
A weighted web $(D,A,B,w)$ is called {\em bipartite} if $V(D) = A \cup B$
and all the edges in $D$ are from $A$ to $B$.
\end{definition}

The following lemma is easy to prove:


\begin{lemma}\label{hetzbipartite}
  If $f$ and $g$ are waves in a bipartite web, then
  $TER(f\hetz g)\cap B=(TER(f)\cup TER(g))\cap B$.
\end{lemma}

We can now use our new machinery to reformulate \Cr{vertexmfmc}:

\begin{conjecture}\label{looseimplieslinkable}
A loose weighted web is linkable.
\end{conjecture}

\begin{lemma} \label{looseimplieslinkableimpliesmain}
\Cr{looseimplieslinkable} implies
\Cr{vertexmfmc} and hence \Cr{mfmcCheat}.
\end{lemma}
\begin{proof} Let $\Gamma$ be a weighted web, and let $f$ be a $(\le)$-maximal
wave in $\Gamma$. Let  $T=\ce(TER(f))$, and let $h$ be a trimming
of $f$. Clearly, $\Gamma/f$ is loose.
Assuming \Cr{looseimplieslinkable}, there exists a
linkage $g$ in $\Gamma/f$. Then, $k=h+g$ is a web-flow,  $T
\subseteq SAT(k)$ and $T$ is $A$-$B$ separating. Since $supp(g)
\subseteq V(\Gamma/f)$ we have $k(x,y)=0$ for every pair of vertices $x,y$ with $x \in V\setminus RF^\circ(T)$ and $y \in RF(T)$, which proves that $T$ is
orthogonal to $k$.
\end{proof}

\section{Attainability of flow values in infinite networks} \label{sec:attain}

In this section we return to flows in networks, rather than
web-flows. Our aim is to prove a result which will serve as a main
ingredient in the proof of \Cr{vertexmfmc} for countable weighted webs (\Tr{main}), and which seems to be of independent interest:

\begin{theorem}\label{attained}
In a countable network $\Delta$ where $d^-_c(x) < \infty$ (i.e.\ the sum of the capacities of the edges pointing to $x$ is finite) for
every vertex $x$, there exists a flow $f$ such that
$|f|=sup\{|g|: \text{$g$ is a flow in $\Delta$} \}$ and $d^-_f(x)
\le |f|$ for every vertex $x$. In particular,
if the values of flows in $\Delta$ are unbounded, then there exists a flow
of infinite value.
\end{theorem}

\begin{definition}\label{residue}
Let $f$ be a flow in a network $\Delta=(D,c,s,t)$ that contains no pair of edges  with the same endvertices but opposite directions. The {\it
residual network} $RES(\Delta,f)$ of $\Delta$ and $f$ is the network $(D',c_R,s,t)$
where $D'$ is the digraph obtained from $D$ by adding an edge $(u,v)$ for every edge $(v,u)\in E(D) \setminus (OUT(s) \cup IN(t) )$, and where $c_R$ is defined by letting, for
every edge $(x,y) \in E(D)$, $c_R(x,y):=c(x,y)-f(x,y)$ and $c_R(y,x):=f(x,y)$.
\end{definition}

For a function $g$ on the edge-set of $RES(\Delta,f)$ let $f\oplus
g$ denote the function $h$ on the edge-set of $\Delta$ defined by
$h(x,y)=f(x,y)+g(x,y)-g(y,x)$. The following is a straightforward corollary of the definitions:
\begin{lemma}\label{flowinres}
Let $f$ be a flow in $\Delta$ and let $g$ be a flow in $RES(\Delta,f)$. Then
$f\oplus g$ is a flow in $\Delta$, with $|f\oplus g|=|f|+|g|$.
\end{lemma}

The following result shows that it is possible to clean up a flow
from cycles and a current coming from infinity without reducing
its value.

\begin{lemma}\label{totallyacyclic}
If $g$ is a flow in $\Delta$ of finite value then there exists a flow
$h \le g$ such that $|h|=|g|$ and $d_h^-(x) \le |h|$ for every
vertex $x$.
\end{lemma}

\begin{proof}
We propagate the desired flow $h$ from $s$. We will define $h$ recursively in infinitely many steps, in each step considering a vertex and adjusting its out-degree to its in-degree, and then removing any cycles we created in doing so. But as subsequent steps might change the in-degree of a vertex we already considered, we will have to return to each vertex infinitely often.

Formally, let $v_1, v_2, \ldots$ be a sequence in which each vertex in $V(\Delta) \setminus \{s,t\}$ appears infinitely often. We will recursively define sequences $(h_i)$, $(h^+_i)$ and $(h^-_i)$ of functions on $E(\Delta)$. Intuitively, $h^+_i$ differs from $h^+_{i-1}$ in that it makes the out-degree of $v_i$ equal to its in-degree, while $h^-_i$ is the sum of some unwanted cycles in $h^+_i$. Subtracting the two we obtain the functions $h_i$ that will converge to the desired flow $h$.

While defining these sequences we will make sure that the following conditions are satisfied for all $i\in\N$:
\begin{enumerate}
\item\label{enum:increasing}
  If $i>0$, then $h^+_{i-1}(e)\le h^+_i(e)$ and $h^-_{i-1}(e)\le h^-_i(e)$ for every edge $e$;
\item\label{enum:bounded}
  $h^-_i(e) \le h^+_i(e) \le g(e)$ for every edge $e$;
\item\label{enum:nocycles}
  the support of $h_i:=h^+_i-h^-_i$ does not contain cycles;
\item\label{enum:nonewflow}
  $d^+_{h^+_i}(v) \le d^-_{h^+_i}(v)$ for every $v\in V(\Delta)\setminus\{s,t\}$; and
\item\label{enum:minusflow}
  $d^+_{h^-_i}(v) = d^-_{h^-_i}(v)$ for every $v\in V(\Delta)\setminus\{s,t\}$.
\end{enumerate}
We start by defining $h^+_0=g$ on $OUT(s)$ and $h^+_0=0$ on all other edges, and $h^-_0=0$.
Clearly, conditions~\eqref{enum:bounded}--\eqref{enum:minusflow} are satisfied for $i=0$.
 For $i=1,2,\ldots$, assume that $h^+_j$ and $h^-_j$ have already been defined for every $j<i$
 and satisfy~\eqref{enum:increasing}--\eqref{enum:minusflow}.

We define $h^+_i$ first. If $d^+_{h^+_{i-1}}(v_{i})<d^-_{h^+_{i-1}}(v_{i})$,
then give each edge $e$ in $OUT(v_{i})$ a value $h^+_i(e)$ with $h^+_{i-1}(e) \le h^+_i(e) \le g(e)$ so that
 (having considered all edges in $OUT(v_i)$) $d^+_{h^+_i}(v_{i})=d^-_{h^+_{i-1}}(v_{i})$ holds;
  this is possible since $d^+_{g}(v_{i})=d^-_{g} (v_{i})$ and $h^+_{i-1} (e)\leq g(e)$ for every $e\in E(\Delta)$
   by condition~\eqref{enum:bounded}. For every edge $e$ in $E(\Delta) \setminus OUT(v_i)$ let $h^+_i(e)=h^+_{i-1}(e)$.
   Clearly, conditions~\eqref{enum:increasing}--\eqref{enum:minusflow} are not violated.

Next we define $h^-_i$. The function $h'_i:=h^+_i-h^-_{i-1}$ is
non-negative since $h^+_i \ge h^+_{i-1} \ge h^-_{i-1}$ by~\eqref{enum:increasing}
and~\eqref{enum:bounded}. If $supp(h'_i)$ contains any cycles, then let $C_1,C_2,\ldots$ be a
(possibly infinite) enumeration of those cycles. (The $C_j$ are not necessarily pairwise edge disjoint.)
We are going to remove all cycles from $supp(h'_i)$ by performing infinitely many steps (within step $i$),
 in each step $j$ eliminating the cycle $C_j$ from $supp(h'_i)$. For every edge $e$ we denote by $h^j_i(e)$
  the value that has to be subtracted from $h'_i(e)$ in order to eliminate the cycles $C_1,\dotsc,C_j$.
  To begin with, let $h^0_i(e)=0$ for every $e\in E(\Delta)$. For $j=1,2,\ldots$, if $C_j$ is a
  cycle in $supp(h'_i-h^{j-1}_i)$ then, for every edge $e \in E(C_j)$, add to $h^{j-1}_i(e)$ the
   value
\begin{equation*}
  \min\{h'_i(e)-h^{j-1}_i(e)\mid e \in E(C_j)\}
\end{equation*}
to obtain $h^j_i(e)$; let $h^j_i(d)=h^{j-1}_i(d)$ for every other edge $d$.

Having treated all cycles $C_j$, define $h^-_i(e):= h^-_{i-1}(e) + lim_j h^j_i(e)$ for every $e\in E(\Delta)$;
this is well defined since $h^j_i(e)$ is monotone increasing with $j$ and bounded by $h'_i(e)$.
 It is not hard to see that conditions~\eqref{enum:increasing}--\eqref{enum:minusflow} are satisfied.



By~\eqref{enum:increasing} and~\eqref{enum:bounded} the sequences $h^+_i(e)$, $h^-_i(e)$ and $h_i(e)$ converge
for every edge $e$;
we let $h^+(e):=\lim_i h^+_i(e)$, $h^-(e):=\lim_i h^-_i(e)$,
and $h(e):=\lim_i h_i(e) = h^+(e)-h^-(e)$. By~\eqref{enum:bounded} we have $h^+(e)\le g(e)$.
Since for every vertex $v\in V(\Delta)\setminus \{s,t\}$ we have $d^+_{h^+_i}(v)=d^-_{h^+_i}(v)$
for infinitely many $i$, we have $d^+_{h^+}(v)=d^-_{h^+}(v)\le d^-_g(v) < \infty$.
By~\eqref{enum:increasing},~\eqref{enum:bounded} and~\eqref{enum:minusflow} we
have $d^+_{h^-}(v)=d^-_{h^-}(v) \le d^-_{h^+}(v)$. Hence $h$ is non-negative and $d^+_h(v)=d^-_h(v)$,
and therefore $h$ is indeed a flow.

Since $IN(s)=\emptyset$, for every edge $e\in OUT(s)$ no cycle
considered in the construction of $h^-$ contained $e$, hence
$h^-(e)=0$ and $h(e)=h^+(e)=h^+_0(e)=g(e)$. Therefore $|h|=|g|$.
Since $h \le h^+ \le g$, all that remains to prove is that
$d^-_h(v)\le|h|$ for every vertex $v$.

Since $|g|$ is finite, $\sum_{e \in E}h_i(e)$ is finite for every
$i < \omega$ by the construction of the functions $h_i$. Let $x
\in V$ and let $X$ be the set of all vertices from which $x$ is
reachable via $supp(h_i)$. Note that $h_i(x,y)=0$ for every vertex
$y \in X \setminus \{x\}$, since otherwise there would exist a
cycle in $supp(h_i)$. Thus, since $d^-_{h_i}(y) \ge d^+_{h_i}(y)$
for every $y \in V \setminus\{s\}$, we have
\begin{equation*}
  \sum_{y \in  X \setminus \{x\}} (d^-_{h_i} (y) - d^+_{h_i} (y)) \ge -d^+_{h_i}(s) = -|h_i|.
\end{equation*}
Since $\sum_{e \in E}h_i(e)$ is finite, we have
\begin{equation*}
  \sum_{y \in  X \setminus \{x\}} (d^-_{h_i}(y) - d^+_{h_i}(y)) = h_i[E(V\setminus(X-x),X-x)] -h_i[E(X-x,V\setminus(X-x))].
\end{equation*}
By the choice of $X$, we have
$h_i[E(V\setminus(X-x),X-x)]=0$ and
\begin{equation*}
  h_i[E(X-x,V\setminus(X-x))]\ge h_i[E(X-x,x)]=d^-_{h_i}(x).
\end{equation*}
This yields $-|h_i| \le \sum_{y \in X-x} (d^-_{h_i}(y) - d^+_{h_i}(y))
\le -d^-_{h_i}(x)$ and hence $d^-_{h_i} (x) \le |h_i| = |h|$. Since $h=\lim_i h_i$, we have $d^-_{h} (x) \leq |h|$.
Since $x$ was chosen arbitrarily, this completes the proof.
\end{proof}

\begin{proof}[Proof of \Tr{attained}]
Without loss of generality, we may assume that $OUT(s)$ consists
of a single edge. Let $\alpha=sup\{|g|:\text{$g$ is a flow in
$\Delta$}\}$. If $\alpha=\infty$, we may just choose flows $f_i$
with $|f_i|= 2^i$, and let $f= \sum \frac{f_i}{2^i}$ which, then,
is a flow with $|f|=\infty$.

So assume that $\alpha$ is finite. Define inductively flows
$f_i$ with $|f_i|=(1-(1/2)^i)\alpha$ as follows. Let $f_0 \equiv 0$.
For every $i>0$, let $g_i$ be a flow in $RES(\Delta,f_{i-1})$ such
that $|g_i| = \frac{1}{2}(\alpha-|f_{i-1}|)=(1/2)^i$ (as $|OUT(s)|=1$, every flow in $\Delta$ of value $|f_{i-1}|
+\frac{1}{2}(\alpha-|f_{i-1}|)$ yields such a flow). By
\Lr{totallyacyclic} there exists a flow $k_i \le g_i$ such
that $d^-_{k_i}(x)\le|k_i|=|g_i|$ for every vertex $x$. By
\Lr{flowinres}, $f_i:=f_{i-1}\oplus k_i$ is a flow of the desired value.

By the choice of the flows $k_i$, the values $f_{i-1}(e)$ and $f_i(e)$
differ by at most $(1/2)^i$ for each edge $e$. Hence the values $f_i(e)$
converge for every $e$; let $f(e)= \lim_i f_i(e)$. It is easy to check
that $f$ is a flow. Further, every vertex $x$ satisfies $d^-_f(x)\le
\alpha$. Since $|f|=\lim_{i}|f_i|= \alpha$, this proves the
theorem.
\end{proof}

We shall use \Tr{attained} twice, and in both cases we
shall use it with the roles of the source and the sink reversed.
Still, we chose this formulation since it is more natural.

\section{Orthogonal pairs in countable networks}
\label{sec:main}

The main result of this section is

\begin{theorem}\label{main}
In any countable network there
exists an orthogonal pair of a cut and a flow.
\end{theorem}

By \Lr{looseimplieslinkableimpliesmain}, in order to prove \Tr{main} it suffices to show:

\begin{theorem}\label{main2}
A loose countable weighted web is linkable.
\end{theorem}

In order to prove \Tr{inf-menger} for the special case of digraphs containing no infinite paths \cite{roneuropean}, or no infinite outgoing paths \cite{ronyellowbook}, it was possible, and useful, to reduce the problem to the special case of bipartite digraphs. Here we are going to use a similar reduction in order to deduce \Tr{main2} from its bipartite counterpart. This is done by the following transformation.

Let $\Delta = (D,A,B,w)$ be a weighted web. We define a bipartite weighted web
$\Gamma = (D',A',B',w')$ in the following way. For each vertex $v
\in V(D) \setminus A $ we introduce a new vertex $v_B$. For
each vertex $v \in V(D) \setminus B $ we introduce a new
vertex $v_A$. We set $A'= \{v_A \mid v \in  V(D) \setminus B
\}$, $B'= \{v_B \mid v \in  V(D) \setminus A \}$, $V(D') = A' \cup
B'$, $E(D') = \{(u_A, v_B) \mid (u,v) \in E(D)\} \cup \{(v_A,v_B)
\mid v \in V(D) \setminus (A \cup B)\}$, $w'(v_A) = w(v)$ for $v \in
V(D) \setminus B$, and $w'(v_B) = w(v)$ for $v \in V(D) \setminus
A$.

If $S$ is a separating set in $\Gamma$ then, defining $A_S = \{v
\mid v_A \in S\}$ and $B_S = \{v \mid v_B \in S\}$, it is
straightforward to check that $S' = (A_S \cap B_S) \cup (A \cap
A_S) \cup (B \cap B_S)$ is a separating set in $\Delta$.
Moreover, waves in $\Gamma$ induce waves in $\Delta$. Indeed, given a wave $f$ in $\Gamma$ with $TER(f)=S$, define the function $f'$
on $E(D)$ by $f'(u,v) = f(u_A,v_B)$. We have:

\begin{lemma}\label{wavetowave}
  $f'$ is a wave in $\Delta$ with $TER(f')=S'$.
\end{lemma}

\begin{proof}
  Let us first prove that $f'$ is a current. To this end, we only have
  to show that $d^+_{f'}(v) \le d^-_{f'}(v)$ for every $v \in V\setminus
  A$. This is clearly true for $v\in B$, so we may assume $v \in V
  \setminus (A \cup B)$. By construction, we have $d^+_{f'}(v)=d^+_f(v_A)
  -f(v_A,v_B)$ and $d^-_{f'}(v)=d^-_f(v_B)-f(v_A,v_B)$. Hence we are
  done if $d^+_f(v_A) \le d^-_f(v_B)$. So let us assume that $d^+_f(v_A)
  >d^-_f(v_B)$. Since $d^+_f(v_A) \le w'(v_A)=w(v)=w'(v_B)$, we have
  $d^-_f(v_B)<w'(v_B)$ and hence $v_B \notin TER(f)$. Likewise, we have
  $d^+_f(v_A) > d^-_f(v_B) \ge 0$ and hence $v_A \notin TER(f)$. Since
  $v_A$ and $v_B$ are connected by an edge, this contradicts the fact
  that $f$ is a wave.

  Now let us prove $TER(f')=S'$. Clearly, we have $TER(f')\cap A =
  A\cap A_S$ and $TER(f')\cap B = B\cap B_S$. Thus, it remains to show
  that $TER(f') \setminus (A \cup B) = A_S \cap B_S$. Let $v \in
  A_S \cap B_S$. Since $v \in A_S$, we have $v_A \in SINK(f)$ and thus
  $v \in SINK(f')$. Finally, we have $v_B \in SAT(f)$, which yields
  $v \in SAT(f')$, since $f(v_A,v_B)=0$. Hence $TER(f') \setminus
  (A \cup B) \supseteq A_S \cap B_S$. Now let $v \in TER(f') \setminus
  (A \cup B)$. Then, $w(v) = d^-_{f'}(v) \le d^-_f(v_B) \le w'(v_B)
  = w(v)$, which means $v \in B_S$ and $d^-_f(v_B)=d^-_{f'}(v)$. The
  latter yields $f(v_A,v_B)=0$. Since $v\in SINK(f')$, we have
  $v \in A_S$.

  Therefore, $TER(f') = S'$. Since $S'$ is an $A$-$B$-separator, $f'$
  is a wave in $\Delta$.
\end{proof}

\begin{lemma}\label{hindrancetohindrance}
  If the zero wave in $\Gamma$ is a hindrance, then the zero wave in
  $\Delta$ is also a hindrance.
\end{lemma}

\begin{proof}
  Suppose the zero wave $f_0$ in $\Gamma$ is a hindrance and let $v_A$
  be a hindered vertex, that is, $w'(v_A) > 0 = d^+_{f_0}(v_A)$ and $v_A
  \notin \ce(TER(f_0))$. In other words, every neighbour $u_B$ of $v_A$
  lies in $TER(f_0)$ and hence satisfies $w'(u_B)=0$. If $v_B$ existed,
  we would have $w'(v_B)=w'(v_A)>0$. Hence $v_B$ does not exist, which
  means that $v \in A$. Further, every neighbour $u$ of $v$ in $\Delta$
  satisfies $w(u)=w'(u_B)=0$, since $u_B$ is a neighbour of $v_A$ in
  $\Gamma$. Therefore, $v \in A\setminus\ce(TER(f'_0))$ and $w(v)=w'(v_A)
  >0$. Hence the zero wave $f'_0$ in $\Delta$ is a hindrance.
\end{proof}

Our next aim is to prove:

\begin{theorem}\label{halllike}
A countable loose bipartite weighted web is linkable.
\end{theorem}

\Tr{halllike}  implies \Tr{main2}. Indeed, Lemmas~\ref{wavetowave} and~\ref{hindrancetohindrance} imply that if
$\Delta$ is loose then so is $\Gamma$. On the other hand, if $f$ is a linkage in $\Gamma$, then the function $f'$ defined above
satisfies $d^+_{f'}(v) = w'(v_A)-f(v_A,v_B) \ge d^-_f(v_B)-f(v_A,v_B)
= d^-_{f'}(v)$ for $v\in V(D)\setminus B$ and $d^+_{f'}(a)=w(a)$ for
$a\in A$. Thus, applying ideas similar to those in the proof of
\Lr{trimming}, we can easily use $f'$ to obtain a linkage of $\Delta$.

The rest of this section will be devoted to the proof of
\Tr{halllike}. Henceforth $\Gamma$ will denote a countable
bipartite weighted web with sides $A$ and
 $B$ and weight function $w$.


\begin{definition} \label{webminuscurrent}
If $f$ is a current in
$\Gamma$ we write $\Gamma-f$ for the weighted web $(D,A,B,w-d_f)$.
\end{definition}



\begin{lemma}\label{nottoobadlyhindered}
Let $\Gamma=(D,A,B,w)$ be a bipartite weighted web and 
let $u$ be a non-negative function on $B$ 
such that $\varepsilon := \sum_{v\in B} u(v)$ is finite. Let $w'$ be the weight function on $V$
defined by $w'\upharpoonright A= w \upharpoonright A$ and
$w'\upharpoonright B = (w \upharpoonright B) -u$. If $\Xi=(D,A,B,
w')$ is $(>\varepsilon)$-hindered then $\Gamma$ is hindered.
\end{lemma}

\begin{proof}
Let $f$ be a $(>\varepsilon)$-hindrance in $\Xi$, and let $a \in A
\setminus  \ce(TER(f))$ be a $(>\varepsilon)$-hindered vertex for $f$,
that is $w(a)-d_f(a)>\varepsilon$. We define a network $\Psi$, as
follows. The vertex set of $\Psi$ is $V(\Gamma) \cup \{t\}$, where
$t$ is a new vertex added (recovering, in fact, the sink vertex of
the network from which the web $\Gamma$ was obtained). The source
vertex of $\Psi$ is $a$, and its sink vertex is $t$. The edges of
$\Psi$ are all edges of $\Gamma$, taken each in both directions,
together with $\{(y,t) \mid y \in B\}$. Its capacity function is
defined by 
$c_\Psi(x,y)=\max(w(x),w(y))+1,
~c_\Psi(y,x)=f(x,y)$ for all $(x,y)\in E(\Gamma)$, and
$c_\Psi(y,t)=u(y)$ for all $y \in B$. By \Tr{attained}
(with the roles of the source and
the sink reversed) there exists in $\Psi$ a flow $j$ maximizing
the in-degree of $t$, and satisfying $d_j^+(a)\le d_j^-(t)\le
\varepsilon$. Note that for $x\in\ce(TER(f))\cap A$, we have
$c_\Psi(e)=0$ for each $e\in IN(x)$ and thus $d_j^-(x)=d_j^+(x)=0$.

Call a vertex $r \in V$ {\it reachable} (from  $a$) if there
exists a path $P$ from $a$ to $r$ in $\Psi$ such that
$c_\Psi(e)-j(e)>0$ for all $e \in E(P)$. 
Note that $c_\Psi(e)-j(e)>0$ for each $A$--$B$~edge $e$. Hence, if
a vertex in $A$ is reachable then so are all its neighbours in $B$.
Let $g$ be the flow defined by letting $g(e)=0$ if $e$ has at least one unreachable endpoint and $g(e)=(f\oplus j)(e)$ otherwise.
We shall show that $g$ is a wave in $\Gamma$. First note that $g$
is a current since $d_g(x)\le d_{f\oplus j}(x)=d_f(x)+j(x,t)\le
w(x)$ for every $x\in V(\Gamma)$. Suppose, for contradiction, that
$TER(g)$ is not $A$-$B$~separating, in which case there exists an edge
$(x,y)$ such that neither $x$ nor $y$ are in $TER(g)$.
Since $x \not \in TER(g)$ it is reachable and so is $y$; indeed,
if $x$ was unreachable, we would have $x\in SINK(g)$ by definition
of $g$, and hence $x \in TER(g)$. Thus, there exists a path $P$
from $a$ to $y$ such that $c_\Psi-j$ is positive on the edges of
$P$.

If $x\in\ce(TER(f))$, we have $d^-_j(x)=d^+_j(x)=0$. This yields
$d_g^+(x)=0$ and thus $x\in TER(g)$, a contradiction. Thus, since
$f$ is a wave, $y\in TER(f)$. Since $y\notin TER(g)$, it is
saturated by $f$ but not by $g$. This means that
$c_\Psi(y,t)-j(y,t)>0$.
Thus the flow $j$ in $\Psi$ can  be augmented along $P$, by
adding some small number $\zeta$ on all edges of $P$
and on $(y,t)$. 
This contradicts the maximality of $d^-_j(t)$.

Therefore, $g$ is a wave in $\Gamma$.
Since  $d_j^+(a)\le d_j^-(t)\le \varepsilon$, we have $d_g(a)
<w_\Gamma(a)$. Thus $a$ witnesses the fact that $g$ is a hindrance
in $\Gamma$, which proves the lemma.
\end{proof}


\Lr{nottoobadlyhindered} and Observation~\ref{deletingaisgoodforyou}
imply:

\begin{corollary}\label{bighindrance}
If $g$ is a current in $\Gamma$ with $\sum_{v\in B} g(v)=
\varepsilon$, and if $\Gamma -g$ is $(>\varepsilon)$-hindered,
then $\Gamma$ is hindered.
\end{corollary}

If $\Gamma=(D,A,B,w)$ is a weighted web and $g$ a real function on the vertices of $\Gamma$ such that $g(v)\leq w(v)$ for every $v \in V(D)$, we write $\Gamma - g$ for the weighted web $(D,A,B,w-g)$.

\begin{lemma}
\label{subtract} Let $\Omega=(D,A,B,w)$ be a loose bipartite
weighted web, and let $b$ be an element of $B$ with $w(b)>0$. Then
there exists $\varepsilon>0$ such that $\Omega -\varepsilon
\chi(\{b\})$ is unhindered.
\end{lemma}

Recall that  $\Omega -\varepsilon \chi(\{b\})$ is obtained from
$\Omega$ by reducing the weight $w$ on $b$  by $\varepsilon$.

\begin{proof}
Without loss of generality we may assume that $w(b)\ge 1$. This
means that $\Omega-{\frac{1}{n}}\chi(\{b\})$ is defined for all
positive integers $n$. Suppose, for contradiction, that
$\Omega-{\frac{1}{n}}\chi(\{b\})$ contains a hindrance $g_n$ for
every $n=1,2,3...$. Clearly, $b \in TER(g_n)$, since otherwise $g_n$ would
be a hindrance in $\Omega$. We define a wave $g_\omega$ in
$\Omega-\chi(\{b\})$ as follows. First, for every $i$, let $\tilde
g_i$ be a wave in $\Omega-\chi(\{b\})$ obtained from $g_i$ by
reducing its value on some edges at $b$ so that $d^-_{\tilde g_i}(b)
=w(b)-1$. Then, let $f_n=g_1\hetz\tilde g_2 \hetz \ldots \hetz\tilde g_n$ and let $g_\omega=\sup f_n$. By \Lr{supwave}, $g_\omega$ is a wave in
$\Omega-\chi(\{b\})$.

Now let $h_n=g_n\hetz g_\omega$. It is easy to check that $h_n$ is
a wave in $\Omega-{\frac{1}{n}}\chi(\{b\})$, even though $g_\omega$
is not: $g_\omega \upharpoonright ((\Omega-{\frac{1}{n}}\chi(\{b\}))/g_n)$ is a wave in
$(\Omega-{\frac{1}{n}}\chi(\{b\}))/g_n$ and hence $h_n$ is a wave, by \Lr{hetziswave}.
Let $T=TER(g_\omega)\cap B$ and let $S=A\setminus RF(T)$. Then, by
\Lr{hetzbipartite}, $T\supset TER(g_n)\cap B$ for all $n$,
and hence $T=TER(h_n)\cap B$ for all $n$. The waves $h_n$ all play
in the same arena - the web induced on $(A \setminus S) \times T$.

Similarly with \Lr{nottoobadlyhindered}, we can define a
network $\Psi$ with sink $b$ and source $s$, where $s$ is a new
vertex added, joined to all vertices in $A\setminus S$. In $\Psi$,
we can apply \Tr{attained} to the flows $h_2-h_1,h_3-h_1,
\ldots$, to deduce that there exists a current $k$ in $\Psi$ of
value $1$. Then, $h_1\oplus k$ is a current in $\Omega$ saturating
all vertices in $T$, and is thus a non-zero wave in $\Omega$,
contradicting the fact that $\Omega$ is loose.
\end{proof}

We shall use \Lr{subtract} for our next lemma:

\begin{lemma}
\label{saturation} Let $\Omega=(D,A,B,w)$ be a loose bipartite weighted web,
and let $a$ be any element of $A$. Then, there exists a
current $f$ such that $d_f(a)=w(a)$ and $\Omega-f$ is loose.
\end{lemma}
\begin{proof}
We may assume that $w(a)>0$ since otherwise we could choose $f\equiv 0$.
We choose recursively vertices $y_\theta \in B$,  flows $f_\theta$
and networks $\Omega_\theta$, for countable ordinals $\theta$, as
follows. Since $w(a)
=0$ and  since $\Omega$ is unhindered, there exists an edge
$(a,y_0) \in OUT(a)$, such that $w(y_0)>0$. By \Lr{subtract}  and
Observation~\ref{deletingaisgoodforyou} we can find $\varepsilon_0>0$ such
that $\Omega-{\varepsilon_0}\chi(\{a,y_0\})$ is unhindered. Let $k_0$
be a maximal wave in $\Omega-{\varepsilon_0}\chi(\{a,y_0\})$. Define
$f_0={\varepsilon_0}\chi(\{(a,y_0)\})+k_0$.
Since $k_0$ is maximal, $\Omega-f_0$ is loose; for if $\Omega-f_0$
contained a wave $g$ which is non-zero or a hindrance, $k_0+g$ would
be a wave in $\Omega-f_0$, contradicting either the maximality of $k_0$
or the fact that $\Omega-f_0$ is unhindered. Let
$\Omega_1=\Omega-f_0$. If $w_{\Omega_1}(a)$, i.e.\ the capacity of $a$ in $\Omega_1$, is greater than $0$, then
there exists $(a,y_1)\in OUT(a)$ with $w_{\Omega_1}(y_1)>0$. Thus we can
find $\varepsilon_1>0$ such that
$\Omega_1-{\varepsilon_1}\chi(\{a,y_1\})$ is unhindered. Taking a
maximal flow $k_1$ in $\Omega_1-{\varepsilon_1}\chi(\{a,y_1\})$ and
defining $f_1={\varepsilon_1}\chi(\{(a,y_1)\})+k_1$, the weighted web
$\Omega_1-f_1=\Omega-f_0-f_1$ is then loose.

We continue this way transfinitely until either the capacity of $a$ has been reduced to $0$ or we have obtained a hindered weighted web. For each ordinal $\alpha$ write
$f_\alpha=\sum_{\theta<\alpha}f_\theta$. For successor ordinals,
the currents $f_\alpha$ and the weighted webs $\Omega_\alpha$ are
defined as exemplified above. For limit ordinals $\alpha$ define
$\Omega_\alpha=\Omega-f_\alpha$. We wish to show that
$\Omega_\alpha$ is unhindered for every $\alpha$.  By the
construction, this is automatically true for successor ordinals
$\alpha$. Thus we only have to show:

\begin{assertion}
$\Omega_{\alpha}$ is unhindered for all limit countable ordinals
$\alpha$.
\end{assertion}

\begin{proof}
The proof is by induction on $\alpha$. Let $\alpha$ be a limit
ordinal, and assume that $\Omega - f_\nu$ is unhindered for all
limit ordinals $\nu < \alpha$. Clearly, hindrances cannot appear
at non-limit ordinals, and thus we may assume that $\Omega -f_\nu$
is unhindered for all $\nu < \alpha$. Assume, for contradiction,
that there exists a hindrance $h$ in $\Omega_\alpha$. Let $z \in
A$ be a hindered vertex and let $\delta =
w_{\Omega_\alpha}(z)-d_{h}(z)$. Since $\sum_{\theta < \alpha}
d_{f_\theta}(a)$ is bounded (by $w(a)$ for instance), there is
some ordinal $\nu$ such that
$\sum_{\nu<\theta<\alpha}d_{f_\theta}(a) < \delta$. In particular,
$\sum_{\nu<\theta<\alpha}\varepsilon_\theta < \delta$. Since
$f_\alpha=f_\nu+\sum_{\nu<\theta<\alpha}{\varepsilon_\theta}\chi\{(a,y_\theta)\}
+ \sum_{\nu < \theta<\alpha}k_\theta$, the current  $\sum_{\nu <
\theta<\alpha}k_\theta+h$ is a $(\ge\delta)$-hindrance in
$\Omega-f_\nu
-\sum_{\nu<\theta<\alpha}{\varepsilon_\theta}\chi\{(a,y_\theta)\}$.
But since $\sum_{\nu<\theta<\alpha}\varepsilon_\theta < \delta$,
this contradicts the fact that $\Omega -f_\nu$ is unhindered by
Corollary~\ref{bighindrance}. This proves the assertion.
\end{proof}

Since $w_{\Omega_{\theta+1}}(a)<w_{\Omega_\theta}(a)$ for every
$\theta$, the process must stop at some countable ordinal
$\alpha$. But this can only happen when $w_{\Omega_{\alpha}}(a)=0$.
Taking $f=f_\alpha$ for $\alpha$ satisfying this condition yields
the lemma. \end{proof}

Applying this lemma recursively, we can now achieve our aim:

\begin{proof}[Proof of \Tr{halllike}]
Enumerate the vertices in $A$ as $a_1,a_2,\ldots$. Applying \Lr{saturation} to $\Delta$ with $a=a_1$ we get
a current $f_1$ in $\Delta$ saturating $a_1$, and having the
property that $\Delta -f_1$ is loose. Using the same lemma
again, we get a current $f_2$ in $\Delta-f_1$ saturating $a_2$
in this weighted web, and such that $\Delta-f_1-f_2$ is loose.
Continuing this way, we find a sequence $f_i$ of currents, where
$f_i$ saturates $a_i$ in $\Delta-\sum_{j<i}f_j$. The current $\sum
f_i$ is then the desired linkage of $\Delta$.
\end{proof}

As already mentioned, \Tr{halllike} implies \Tr{main2}, which in turn implies \Tr{main}.

\section{Mundane flows and attainability}
\label{sec:mundane}

As mentioned in the introduction, \Tr{main} does not generalise \Tr{inf-menger}, since the flow is allowed to contain infinite paths. One could try to generalise \Tr{inf-menger} by only considering flows that do not contain infinite paths:

\begin{definition}
A flow $f$ is {\em mundane} if (seen as a vector in $\R_+^E$) it
can be written as $f=\sum_{i \in I}\theta_i \chi(E(P_i))$, where
$\theta_i$ is a positive real number and $P_i$ is an $s$-$t$ path.
\end{definition}

\begin{problem}\label{orthmundane}
 Does there exist an orthogonal pair of a cut and a mundane flow for every infinite network?
\end{problem}

The results proved so far answer this question for certain networks. A {\it trail} in a network is a directed walk in which no edge appears more than once.

\begin{corollary}\label{mainmundane}
  In every countable network $\Delta=(D,s,t,c)$ that contains no infinite trail, there is an orthogonal pair of a cut and a mundane flow.
\end{corollary}

\begin{proof}
  By the transformation of \Sr{sec:vx}, $\Delta$ yields a weighted web $\Gamma=(D',A,B,w)$. Recall that $V(D')=E(D)$. By \Tr{main2} and \Lr{looseimplieslinkableimpliesmain} there is an orthogonal pair of a separating set $S$ and a web-flow $f$ in $\Gamma$. We may assume $S$ to be essential. We claim that there is a {\it mundane web-flow} $f'\le f$ that is also orthogonal to $S$, where mundane web-flows are defined analogously to mundane flows.

  Since $\Delta$ contains no infinite trails, there are no infinite paths in $\Gamma$; we will use this fact to construct $f'$. Inductively for countable ordinals $i$ we will choose $A$--$B$~paths $P_i$ and positive real numbers $\theta_i$ so that the function $f_i:=\sum_{j\le i}\theta_j\chi(E(P_j))$ is a mundane web-flow with $f_i(e)\le f(e)$ for each edge $e$. Let $i$ be a countable ordinal and assume that $P_j$ and $\theta_j$ have been defined for all $j<i$. Then, since each $f_j$ satisfies $f_j \le f$ by assumption, the function $f_{<i}:=\sum_{j<i}\theta_j\chi(E(P_j))$ is a mundane web-flow with $f_{<i} \le f$. If $f_{<i}(e)=f(e)$ for every $e\in E(A,V(D')\setminus A)$, we terminate the construction and put $f':=f_{<i}$. Otherwise, since $\Gamma$ contains no infinite paths, the support of the web-flow $f-f_{<i}$ contains an $A$--$B$~path $P_i$; let $\theta_i:=\min\{f(e)-f_{<i}(e) \mid e\in E(P_i)\}$. Clearly, $f_i$ is a mundane web-flow with $f_i\le f$. Since $supp(f_i)\subsetneq supp(f_{<i})$ and $\Gamma$ is countable, the construction terminates after countably many steps.

  We thus have a mundane web-flow $f'\le f$ that coincides with $f$ on $E(A,V(D')\setminus A)$. We have to show that $f'$ is orthogonal to $S$. Since $f'\le f$ and $f$ is orthogonal to $S$, it suffices to show that $S\subset SAT(f')$. If $d^-_{f'}(s)<w(s)$ for a vertex $s\in S\setminus A$, then $d^-_{f-f'}(s)>0$. Since $f-f'$ is a web-flow, no vertex in the digraph $\tilde D=(V(D'),supp(f-f'))$ that does not lie in $A\cup B$ has degree $1$. Hence $s$ lies on an $A$--$B$~path in $\tilde D$, or on an infinite path, or on a cycle. By the choice of $f'$, there are no $A$--$B$~paths in $\tilde D$, and $\tilde D$ does not contain infinite paths since $D'$ does not. So $s$ lies on a cycle $C$ in $\tilde D$, which is clearly also a cycle in $supp(f)$. Since $S$ is essential we have $s\in RF(S)\setminus RF^\circ(S)$, and hence $C$ contains an edge $e$ from $V(D')\setminus RF^\circ(S)$ to $RF(S)$. But then $e\in supp(f)$ and thus $f(e)>0$, contradicting the fact that $f$ is orthogonal to $S$.

  We have shown that there is an orthogonal pair of a separating set $S$ and a mundane web-flow $f'=\sum_{i\in I}\theta_i\chi(E(P_i))$ in $\Gamma$. This pair can easily be translated into an orthogonal pair of a cut $F$ and a mundane flow $g$ in $\Delta$: The vertex set $S$ in $D'$ is an edge set in $D$ and it is $s$--$t$~separating in $D$ since it is $A$--$B$~separating in $D'$; hence it contains a cut $F$ in $D$. Every $A$--$B$~path $P_i$ in $D'$ corresponds to an $s$--$t$~trail $P'_i$ in $D$; let $g'$ be the function on $E(D)$ defined by $g':=\sum_{i\in I}\theta_i\chi(E(P'_i))$. It is easy to see that $g'$ is a flow in $\Delta$ orthogonal to $S$ and hence also to $F$. Therefore, each $P'_i$ meets $F$ in precisely one edge. Every $P'_i$ contains an $s$--$t$~path $Q_i$; let $g:=\sum_{i\in I}\theta_i\chi(E(Q_i))$. Then $Q_i$ meets $F$ at the same edge as $P'_i$ does, and hence $g$ is a mundane flow in $\Delta$ orthogonal to the cut $F$.
\end{proof}

In the remainder of this section we show that the infimum $\sigma$ of the
capacities of the $s$--$t$~cuts in a network equals the supremum $\tau_m$ of the values of the mundane flows. Moreover, we show that $\sigma$ is attained by some cut but $\tau_m$ need not be attained by any mundane flow.

\begin{definition}
  Given a countable network $\Delta=(D,c,s,t)$, let
  \begin{align*}
    \sigma&:=\inf\{c[F] : F\text{ is an $s$--$t$~cut}\},\text{ and}\\
    \tau_m&:=\sup\{|f| : f\text{ is a mundane flow in }\Delta\}.
  \end{align*}
\end{definition}

\begin{theorem} \label{sigtau}
  Let $\Delta=(D,c,s,t)$ be a countable network. The following statements hold:
  \begin{enumerate}
  \item
    $\Delta$ has an $s$--$t$~cut $F$ of minimal capacity $\sigma$, and
  \item
    $\tau_m=\sigma$.
  \end{enumerate}
\end{theorem}

\begin{proof}
For every positive integer
$i$, let $c_i$ be the function obtained from $c$ by cutting off
everything behind the $i$th decimal; formally, $c_i(e)=\lfloor
10^ic(e)\rfloor / 10^i$. In the network $\Delta_i=(D,c_i,s,t)$, all capacities
are multiples of $10^{-i}$, hence we can use \Tr{inf-menger} to find
an orthogonal pair of a mundane flow $f_i$ and a cut $F_i$ in $\Delta_i$.
Since $c_i\le c$, $f_i$ is also a flow in $\Delta$. This yields $c_i[F_i]=
|f_i|\le\tau_m$. We will use the cuts $F_i$ to construct a cut $F$ with
capacity $\tau_m$.

First, enumerate all edges in $E(D)$ as $e_1,e_2,\ldots$. Then, inductively for every positive integer $i$, if there is an integer $m$ such that $m>j_l$
for all $l<i$ and the set $\{e_{j_1},\ldots,e_{j_{i-1}},e_m\}$ is contained in
infinitely many of the cuts $F_1,F_2,\ldots$, then let $j_i$ be the smallest such integer $m$. If no such $m$ exists, then stop.

If $j_i$ exists for all
$i$, we end up with a set $F'=\{e_{j_1},e_{j_2},\ldots\}$ of edges. Now choose a subsequence of $F_1,F_2,\ldots$ as follows: For every positive
integer $i$, let $k_i$ be the smallest integer such that $k_i>k_l$ for all
$l<i$ and the set $\{e_{j_1},\ldots,e_{j_i}\}$ is contained in $F_{k_i}$.

If for some $i$ there is no $j_i$ as desired, we end up with a finite set
$F'=\{e_{j_1},\ldots,e_{j_{i-1}}\}$ and we choose $F_{k_1},F_{k_2},\ldots$
to be the subsequence of $F_1,F_2,\ldots$ consisting of all cuts that contain
$F'$.

In both cases, every edge $e_l$ that is contained in infinitely many of
the cuts $F_{k_1},F_{k_2},\ldots$ is contained in $F'$, since it must have been
chosen as $e_{j_i}$ at some step $i$. We claim that $c[F']\le\tau_m$. Indeed, for every $\varepsilon>0$, there is a
finite subset $F''$ of $F'$ with $c[F'']\ge c[F']-\frac12\varepsilon$. For
sufficiently large $i$, we have $c_i[F'']\ge c[F'']-\frac12\varepsilon$ and
thus $c[F']\le c_i[F'']+\varepsilon\le \tau_m+\varepsilon$. With $\varepsilon\to
0$, this yields $c[F']\le\tau_m$. We further claim that $F'$ separates $s$ from
$t$. Indeed, let $P$ be an $s$--$t$~path. Since $P$ is finite and $F_{k_1},F_{k_2},\ldots$ infinite, $P$ contains an edge that is contained in
infinitely many of the cuts $F_{k_1},F_{k_2},\ldots$, and is thus contained
in $F'$, so $F'$ meets every  $s$--$t$~path. Therefore, $F'$ contains a cut $F$ which, then, satisfies $c[F]\le\tau_m$.

This shows that $\sigma\le c[F]\le \tau_m$. Combining with the trivial inequality $\tau_m \le\sigma$ we obtain the required result.
\end{proof}

The remaining question is whether there is always a mundane flow of value $\tau_m$. The following example shows that this is not the case, providing a negative answer to Problem~\ref{orthmundane}.

\begin{example}\label{ex:counterex}
We construct a locally finite network in which there is no mundane flow of
maximal value. We start with a disjoint union of (directed) paths $Q_i=x_0^i
x_1^ix_2^ix_3^i$,~$i=1,2,\ldots$. For every positive integer $k$, let each
edge $e$ on any path $Q_i$ with $2^{k-1}\le i\le 2^k-1$ have capacity $c(e)=
1/2^k$. Further, for each such $k$ and $i$, we attach the paths $Q_{2i}$ and
$Q_{2i+1}$ to $Q_i$ by adding the edges $(x_0^i,x_0^{2i})$, $(x_3^{2i},x_2^i)$
(to attach $Q_{2i}$), $(x_1^i,x_0^{2i+1})$, and $(x_3^{2i+1},x_3^i)$ (to attach
$Q_{2i+1}$). Let each such edge $e$ have capacity $c(e)=1/2^k$. We denote the
resulting digraph by $D$. The definition of the network $\Delta=(D,c,s,t)$ is
completed by choosing $s=x_0^1$ and $t=x_3^1$ (see Figure~\ref{fig:counterex}).

\showFig{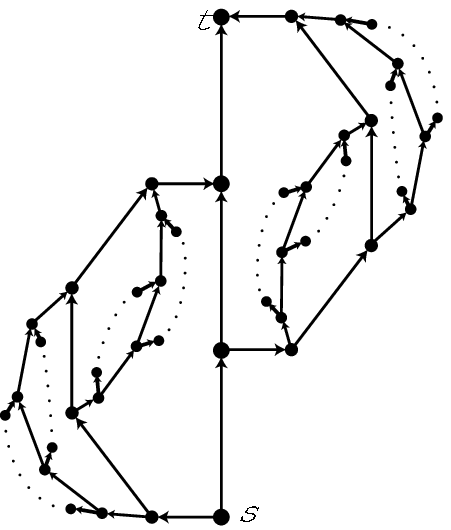}{.5}{fig:counterex}{A locally finite network with no mundane flow of maximal value}

Clearly, $D$ is locally finite (in fact it has maximum degree $3$). For every
positive integer $k$, there exists a mundane flow of value $1-1/2^k$:
It is easy to see that for each positive integer $i$, there is exactly one
$s$--$t$~path that contains $Q_i$; denote it by $P_i$. Then $f_k:=
\sum_{i=1}^{2^{k}-1}\frac1{2^k}P_i$ is a mundane flow of value $1-1/2^k$.
This shows that $\tau_m\ge 1$.

We claim that there is no mundane flow in $\Delta$ that has value $1$.
Indeed, suppose for contradiction that $f$ is a mundane flow with $|f|=1$. Let $e:=(x_1^1,x_2^1)$ and $d:=(x_1^1,x_0^3)$. Applying Kirchhoff's first law to $x_1^1$ we obtain $f(d) \leq 1/2 - f(e)$. However, since $F=\{d, (x_2^1,x_3^1)\}$ is an $s$--$t$~cut with $c[F]=1$, $f$ must saturate $F$ and thus $f(d)=1/2$ whence $f(e)=0$ holds. Similarly, we can prove that $f(g)=0$ holds for every edge $g$ of the form $(x_1^i,x_2^i)$. Since these edges form an $s$--$t$~cut we obtain a contradiction to the fact that $f$ is mundane.

\end{example}

\section{Flowing through an end}
\label{sec:ends}

In this section we consider constraints on  flows that are weaker than being  mundane,
in order to allow for flows to flow, in a sense,  through ends of the underlying undirected graph.
As an example look at the flows in Figure~\ref{fig:nonwcr} and Figure~\ref{fig:wcr2}.
The definition of a mundane flow does not distinguish between the two and rejects both.
However, there is an important difference:
The flow in Figure~\ref{fig:nonwcr} disappears in the left end of the graph and comes back from the right one,
while the flow in Figure~\ref{fig:wcr2} just flows through the left end.
In this section we study flows of the second kind. In order to distinguish them formally
from other flows we need an analog of Kirchhoff's first law for ends. In the case of Figure~\ref{fig:wcr2}
 it is possible to say how much flow arrives at the left end and how much flow leaves it,
 but in general this is not possible: look for example at the network in
 Figure~\ref{fig:counterex}. The flows $f_k$ used there have a limit flow $g$.
 Now for every ray $R$ in this network the values of $g$ along $R$ converge to $0$,
 however there is some flow running to infinity and coming back.
 Similarly to the examples in Figure~\ref{fig:nonwcr} and Figure~\ref{fig:wcr2},
 it is possible to construct flows like $g$ where the
 flow does flow out of the same ends it
 flows in (like in Figure~\ref{fig:wcr2}
 and Figure~\ref{fig:counterex}) or it does
 not (like in Figure~\ref{fig:nonwcr}).
 For flows like $g$ it is not clear how to make precise the assertion than the ends satisfy
 Kirchhoff's first law. The following definition accomplishes this task in an elegant way:

\showFig{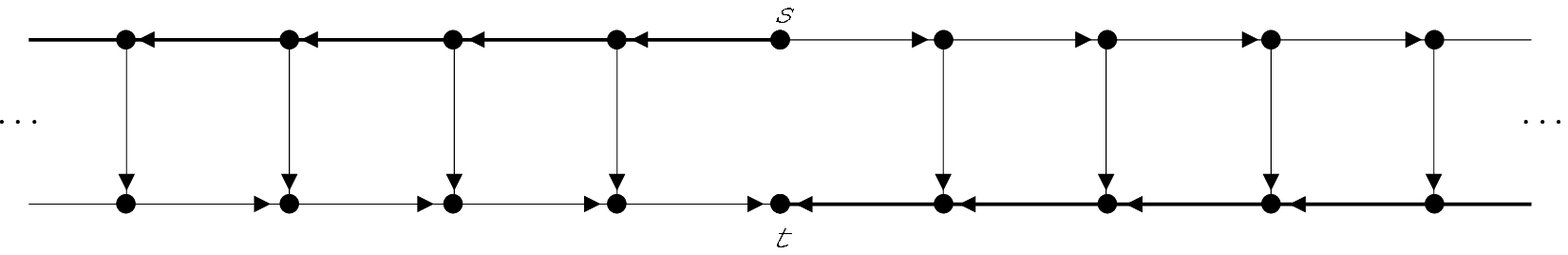}{.92}{fig:nonwcr}{A network and a flow. Thick edges carry a flow of value $1$;
thin edges carry no flow.
This flow flows into the left end of the graph and returns through the right one.}

\showFig{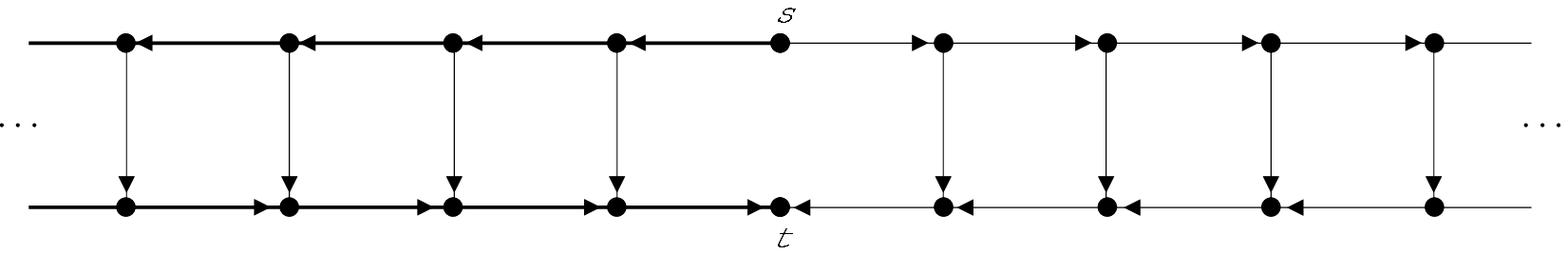}{.92}{fig:wcr2}{A flow flowing through the left end of the graph.}

\begin{definition}We will call a flow $f$ in a network $\Delta=(D,c,s,t)$ \emph{\wcr}
if for every cut $E(S,T)$ in $D$ (where $T=V(D) \setminus S$) with $s\in S$ that
consists of finitely many edges we have

\begin{equation} \label{cond:wcr}
f[E(S,T)] = \begin{cases}
f[E(T,S)]& \text{ if $t\in S$,}\\
f[E(T,S)]+|f|& \text{ if $t\in T$}
.
\end{cases}
\end{equation}
Let $\taw:=sup \{|f| : f \text{ is a \wcr\ flow} \}$, and let \siw\ be the infimum
of the capacities of all  $s$--$t$~cuts consisting of finitely many edges.
\end{definition}

To see why this definition can be thought of as an analog of Kirchhoff's
first law for ends note that in a locally finite network a cut consisting of finitely many edges cannot
separate two rays in the same end. It is easy to check that $g$ as well as the flow in Figure~\ref{fig:wcr2}
is \wcr\ while the flow in Figure~\ref{fig:nonwcr} is not.

\comment{\begin{proposition} \label{prop:wcr}
In every locally finite network $\Delta=(D,c,s,t)$ there is a \wcr\ flow $f$ such that $|f| = \taw$.
\end{proposition}

\begin{proof}
For every edge $e$ in $D$ let $I_e$ be the real interval $[0,c(e)]$, and define the topological
space $X:= \Pi_{e\in E(D)} I_e$. By Tychonoff's theorem $X$ is compact.

Let $f_1,f_2,\ldots$ be a sequence of \wcr\ flows whose values converge to \taw.
Every $f_i$ corresponds to a point $x_i$ in $X$: the point that has value $f_i(e)$ at the coordinate $I_e$
of $X$ for every $e \in E(D)$. Since $X$ is compact, the sequence $x_1,x_2,\ldots$ has an accumulation point $x$,
which determines a function $f: E(D) \to \R$. It is straightforward to check that $f$ is a \wcr\ flow; indeed,
if condition~\eqref{cond:wcr} is violated by $f$ for some finite cut $D$, in particular if Kirchhoff's law is
violated at some vertex, then there is a basic open neighbourhood $O \ni f$ in $X$, chosen by taking a small
enough interval of $I_e$ around $f(e)$ for every $e \in D$, such that every function in $O$ also violates
condition~\eqref{cond:wcr} at $D$. But this cannot be the case since any such $O$ contains some of
the \wcr\ flows $f_i$. Similarly, it is straightforward to check that $|f| = \taw$.
\end{proof}
}

\begin{theorem} \label{prop:wcr}
In every locally finite network $\Delta=(D,c,s,t)$, $\siw = \taw$ holds.
Moreover, there is a \wcr\ flow $f$ such that $|f| = \taw$.
\end{theorem}

\begin{proof}
For every edge $e$ in $D$ let $I_e$ be the real
interval $[0,c(e)]$, and define the topological
space $X:= \Pi_{e\in E(D)} I_e$. By Tychonoff's theorem $X$ is compact.

Pick an $s$--$t$~path $P$ in $D$, and for every $i \in \N$,
let $\Delta_i=(D_i,c_i,s,t)$ be the finite network obtained from $\Delta$ by
contracting each component $C$ of $D - \{x \in V(D) \mid d(x,P)\leq i\}$ to a
vertex $v_C$, and letting $c_i(e) = c(e)$ for every edge in this network. 
(Here $d(x,P)$ stand for the length of the minimal path in the underlying undirected 
graph from $x$ to a vertex of $P$.)
By the MFMC theorem (for finite networks) there is a flow $f_i$ in $\Delta_i$
such that $|f_i|=\sigma_i$, where $\sigma_i$ denotes the minimum capacity of an $s$--$t$~cut in $\Delta_i$.
For every $n$, $f_i$ corresponds to a  point $x_i$ in $X$: the point that has value $f_i(e)$ at
the coordinate $I_e$ of $X$ for every $e \in E(D_i)$ and value $0$ at every other coordinate.
Since $X$ is compact, the sequence $x_1,x_2,\ldots$ has an accumulation point $x$, which determines
a function $f: E(D) \to \R$.

We claim that $f$ is a \wcr\ flow in $\Delta$; indeed, if \eqref{cond:wcr} is violated by $f$ for
some finite cut $B$, in particular if Kirchhoff's law is violated at some vertex,
then there is a basic open neighbourhood $O \ni x$ in $X$, chosen by taking a small
enough interval of $I_e$ around $f(e)$ for every $e \in B$,
such that every function in $O$ also violates \eqref{cond:wcr} at $B$.
 But this cannot be the case since any such $O$ contains some $x_i$ where $i$
 is large enough so that $B$ is a cut in $D_i$.

Similarly, it is not hard to check that $|f|$ is an accumulation point of the
sequence $\{|f_i|\}_{i\in \N}$. Since any cut in some $D_i$ is also a cut in $D$,
 we have $\sigma_i \geq \siw$, and since $|f_i| = \sigma_i$, we obtain $|f| \geq \siw$.
  But $|f|\leq \taw \leq \siw$ by \eqref{cond:wcr}, thus $|f| = \taw = \siw$
\end{proof}

Thus the value $\taw$ is always attained by some \wcr\ flow.
However, $\siw$ does not have to be attained by some finite cut, as shown by the following example.

\begin{example}\label{ex:siw}
Starting with the network of Example~\ref{ex:counterex}, we modify
the capacities of its edges as follows. For every edge $e$ that is
the middle edge $(x_1^i,x_2^i)$ of some path $Q_i$ let $c'(e)=0$; for every other edge $f$, if $c(f)=1/2^k$ then let $c'(f)=1/4^k$. Now the resulting network $\Delta'=(D,c',s,t)$ has $\siw=0$ but the only cut of capacity $0$ is the infinite cut consisting of all the middle edges of the $Q_i$.
\end{example}

Although the definition of a \wcr\ flow allows flows through ends and
forbids flows like the one in Figure~\ref{fig:nonwcr}, there are also
instances of \wcr\ flows that may seem unnatural. Look for example at Figure~\ref{fig:nonscr};
it shows a \wcr\ flow of value $1$ from $s$ to $t$, in a network that
 contains no finite directed $s$--$t$~path. The following definition bans such flows.

\showFig{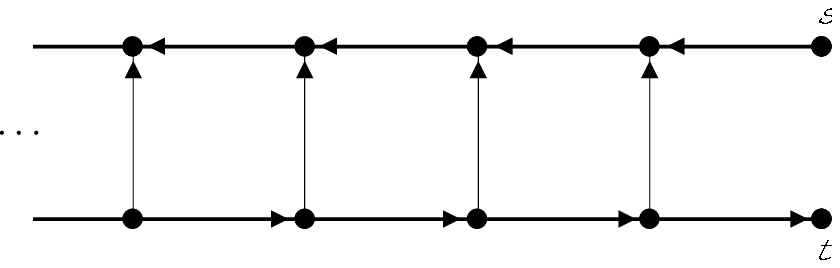}{.7}{fig:nonscr}{A non-zero \wcr\ flow in a network with no finite directed $s$--$t$~path.}

\begin{definition}
We will call a flow $f$ in a network $\Delta=(D,c,s,t)$ \emph{\scr} if it is
\wcr\ and moreover for every $s$--$t$~cut $E(S,T)$ in $D$ we have

\begin{equation} \label{cond:scr}
\begin{split}
|f| + f[E(T,S)]  &\leq c[E(S,T)],\text{ and}\\
f[E(S,T)] &\leq c[E(T,S)] + |f|.
\end{split}
\end{equation}
\end{definition}

Intuitively, the first condition demands that if some flow circumvents
 an infinite $s$--$t$~cut $E(S,T)$, then this circumvention does not
  exceed the amount that could flow through $E(S,T)$ given its capacity
  $c[E(S,T)]$, taking into account that the flow through $E(S,T)$ should also
  compensate for any flow $f[E(T,S)]$ in the inverse direction.
  The second condition demands that if some $s$--$t$~cut carries
  more flow than $|f|$, then the excess is not greater than the amount
   than could go back through the inverse cut.

Let $\tas:=sup \{|f| : f \text{ is a \scr\ flow} \}$.


\begin{theorem}
In every locally finite network $\Delta=(D,c,s,t)$ we have $\sis = \tas$.
 Moreover, there is a \scr\ flow $f$ such that $|f| = \tas$ and an $s$--$t$~cut
 $F$ with $c[F]=\sis$ orthogonal to $f$.
\end{theorem}

\begin{proof}
Since, clearly, every mundane flow is \scr, we have $\tas \geq \tau_m$, and thus,
by \Tr{sigtau} and condition~\eqref{cond:scr}, $\sis = \tas$. Let $f_1,f_2,\ldots$
be a sequence of mundane flows in $\Delta$ whose values converge to $\tam=\tas$.
As in the proof of \Tr{prop:wcr}, for every edge $e$ in $D$ let $I_e$ be the real
interval $[0,c(e)]$, and define the topological space $X:= \Pi_{e\in E(D)} I_e$.
Every $f_i$ corresponds to a point $x_i$ in $X$: the point that has value $f_i(e)$
at the coordinate $I_e$ of $X$ for every $e \in E(D)$. Since $X$ is compact,
the sequence $x_1,x_2,\ldots$ has an accumulation point $x$,
which determines a function $f: E(D) \to \R$.
Similarly with the proof of \Tr{prop:wcr},
it is straightforward to check that $f$ is a \scr\
flow since every $f_i$ is, and that $|f|=\tas$.

Let $F$ be an $s$--$t$~cut with $c[F]=\sis$, which exists by \Tr{sigtau}.
We claim that $f$ saturates $F$. Suppose for contradiction that there is an
 edge $e\in F$ such that $f(e)<c(e) - \epsilon$ for some $\epsilon > 0$.
  Then, there is an infinite subsequence $(f'_i)$ of $(f_i)$ with $f'_i(e)< c(e) - \epsilon$.
  But this means that the $f'_i$ are mundane flows in the network $\Delta'$
   obtained from $\Delta$ by reducing $c(e)$ by $\epsilon$. Thus, $\lim |f'_i|
\leq \tam - \epsilon$ by \Tr{sigtau} since $F$ is a cut of capacity  $\sigma - \epsilon$ in that network.
This contradicts the choice of $(f_i)$, so $f$ saturates $F$ as claimed. Similarly,
it is easy to show that for every $T$--$S$~edge $e$ we have $f(e)=0$, which proves that $f$ and $F$ form
an orthogonal pair.

Suppose now for contradiction, that $|f|<\sis$. Then, the auxiliary network $\Delta'=(D,c',s,t)$
obtained by letting $c'(e)=c(e)-f(e)$ for every $e\in E(D)$ has no cut of zero capacity,
because this would imply $|f|\geq \sis(\Delta)$, and no non-trivial \scr\ flow,
 because this would imply $|f|<\tas(\Delta)$. This however cannot be the case; if $\Delta'$ has  no non-trivial \scr\ flow, then there is no finite directed $s$--$t$~path $P$ in $\Delta'$ such that $c'(e)>0$ for every $e \in E(P)$. But then, letting $S$ be the set of vertices $v$ of $D$ such that there is a finite directed $s$--$v$~path $P$ in $\Delta'$ with $c'(e)>0$ for every $e \in E(P)$, we obtain the cut $E(S, V(D) \setminus S)$ which, clearly, has zero capacity.
\end{proof}

It is possible to consider networks where the source $s$ or sink $t$ or both are ends
 of the underlying undirected graph of a digraph $D$ instead of vertices. An $s$--$t$~flow of value $m$ is, then,
 a function $f$ on $E(D)$ such that $KIR(f)=V(D)$ and moreover, for every finite cut $E(S,T)$
  such that $s$ {\it lives} in $S$ we have $f(E(S,T))=m$ unless $t$ also lives in $S$,
  in which case we have $f(E(S,T))=0$. Here, we say that an end lives in $S$ if one of its rays,
  and thus, since $E(S,T)$ is finite, a subray of any of its rays, is contained in $S$; we also
   say that the vertices of $S$ live in $S$. The interested reader will be able to confirm that
   the results of this section carry over to such networks and flows.

\end{document}